\documentclass[12pt]{amsart}

\usepackage{amssymb,amsmath,amsfonts,amsthm,amsrefs}

\usepackage{pstricks}

\theoremstyle{plain}
\newtheorem{thm}{Theorem}[section]
\newtheorem{lem}[thm]{Lemma}
\newtheorem{prop}[thm]{Proposition}
\newtheorem{cor}[thm]{Corollary}

\theoremstyle{definition}

\newtheorem{rem}[thm]{Remark}
\newtheorem{exa}[thm]{Example}

\def\n{\noindent}
\def\I{\mathop{\rm Inc}\nolimits}
\def\SYT{\mathop{\rm SYT}\nolimits}
\def\cont{\mathop{\rm cont}\nolimits}

\newdimen\Squaresize \Squaresize=18pt
\newdimen\Thickness \Thickness=0.4pt

\def\Square#1{\hbox{\vrule width \Thickness
   \vbox to \Squaresize{\hrule height \Thickness\vss
      \hbox to \Squaresize{\hss#1\hss}
   \vss\hrule height\Thickness}
\unskip\vrule width \Thickness} \kern-\Thickness}

\def\Vsquare#1{\vbox{\Square{$#1$}}\kern-\Thickness}

\def\younglambda#1{
\vbox{\smallskip\offinterlineskip \halign{&\Vsquare{##}\cr #1}}}
\def\vy#1{\hskip2pt\vcenter{\younglambda{#1}}\hskip2pt}

\title[Increasing tableaux and the cyclic sieving phenomenon]%
 {Increasing tableaux, Narayana numbers and an instance of the cyclic sieving phenomenon}
\author{Timothy Pressey, Anna Stokke and Terry Visentin}

\address{University of Winnipeg \\
Department of Mathematics and Statistics \\
Winnipeg, Manitoba \\
Canada  R3B 2E9}

\email{a.stokke@uwinnipeg.ca, t.visentin@uwinnipeg.ca}
\thanks{This research was supported in part by a grant from the Natural
Sciences and Engineering Research Council of Canada.}

\begin{document}

\begin{abstract} We give a counting formula for the set of rectangular increasing tableaux in terms of generalized Narayana numbers.  We define small $m$-Schr\"oder paths and give a bijection between the set of increasing rectangular tableaux and small $m$-Schr\"oder paths, generalizing a result of Pechenik \cite{pechenik}.
Using $K$-jeu de taquin promotion, we give a cyclic sieving phenomenon for the set of increasing hook tableaux.  \end{abstract}
\maketitle

\section{Introduction}

Let $\lambda$ be a partition of a positive integer $N$.  An {\em increasing} tableau $T$ is a $\lambda$-tableau in which both rows and columns are strictly increasing and, if $N-k$ is the largest entry in the tableau, then each $i$ with $1 \leq i \leq N-k$ appears at least once in $T$.  Let $\I_k(\lambda)$ denote the set of increasing $\lambda$-tableaux with maximum value $N-k$ and let $\SYT(\lambda)=\I_0(\lambda)$ denote the set of standard $\lambda$-tableaux with entries in $\{1,2,\ldots,N\}$.   In the first half of the article we focus on increasing tableaux of rectangular shape $\lambda=(n,n,\ldots,n)=n^m$ and will denote the corresponding sets by $\I_k(m \times n)$ and $\SYT(m \times n)$. 

The two-dimensional Catalan numbers enumerate $\SYT(2 \times n)$, the set of standard tableaux with two rows.  In \cite{pechenik}, Pechenik gave explicit bijections between $\I_k(2 \times n)$, small Schr\"oder paths with $k$ diagonal steps and $\SYT(n-k,n-k,1^k)$, giving a formula for the cardinality of $\I_k(2 \times n)$ and showing that the total number of increasing tableaux of shape $2 \times n$ is the $n$th small Schr\"oder number. 

The generalized Narayana numbers $N(m,n,\ell)$ studied in \cite{sulanke3} and \cite{sulankeD} count the $m$-dimensional lattice paths from $(0,0,\ldots,0)$ to $(n,n,\ldots,n)$ lying in the region $\{(x_1,x_2,\ldots,x_m)\mid 0 \leq x_m \leq \cdots \leq x_1\}$ using steps $X_1 = (1,0,\ldots,0), X_2 = (0,1,\ldots,0), \ldots, X_m = (0,0,\ldots,1)$, which have $\ell$ ascents.  An ascent in a path occurs when the path contains consecutive steps $X_iX_{j}$ with $j < i$.  We prove that the cardinality of $\I_k(m \times n)$ is a linear combination of Narayana numbers in Theorem \ref{cardthm}.  An interesting corollary is that $\displaystyle \vert \I_1(m\times n)\vert=\frac{(m-1)(n-1)}{2}\vert \SYT(m \times n) \vert$.

The small $m$-Schr\"oder numbers are given by the sequence $(N_{m,n}(2))_{n \geq 0}$,
where $N_{m,n}(t)=\sum_{\ell=0}^{(m-1)(n-1)}N(m,n,\ell)t^\ell$ is the $m$-Narayana polynomial.
In general, the $m$-dimensional Catalan numbers $(N_{m,n}(1))_{n \geq 0}$ enumerate $\SYT(m \times n)$.
We prove that the small $m$-Schr\"oder number $N_{m,n}(2)$ is equal to the total number of increasing rectangular tableaux
of shape $m \times n$, generalizing Pechenik's result for $\I_k(2 \times n)$.
We define a generalized version of small Schr\"oder paths in $m$-dimensional space, called small $m$-Schr\"oder paths,
and give a bijection between small $m$-Schr\"oder paths and the set of increasing tableaux of shape $m \times n$.

  Let $X$ be a finite set, $C=\langle g \rangle$ a cyclic group of order $N$ that acts on $X$, and $X(q)$ a polynomial with integer coefficients.  The triple $(X,C,X(q))$ is said to exhibit the cyclic
 sieving phenomenon (CSP) \cite{rsw} if for any positive integer $d$,  we have $X(\omega^d)=\vert \{x \in X \mid g^d x =x\}\vert,$
 where $\omega=e^{2\pi i/N}$ is a primitive $N$th root of unity.
A CSP for $\SYT(m \times n)$ was given by Rhoades \cite{rhoades} using classic jeu de taquin promotion and
a $q$-analogue of the Frame-Robinson-Thrall hook length formula \cite{frt}.
Using the $K$-jeu de taquin of Thomas and Yong \cite{thomasyong} and a natural $q$-analogue of a formula
that enumerates $\I_k(2\times n)$, Pechenik gave a CSP for $\I_k(2 \times n)$.
Rhoades~\cite{rhoadesnew} has recently given a representation-theoretic proof of this result.
The natural $q$-analogue of our counting formula for $\I_k(m\times n)$ does not, in general,
serve as a CSP polynomial for the action of $K$-promotion on $\I_k(m \times n)$.
In Section \ref{sec:csp}, we focus on proving a CSP for the set $\I_k(N-r,1^r)$ of increasing hook tableaux
using $K$-promotion and a $q$-analogue of a formula that enumerates $\I_k(N-r,1^r)$.
This polynomial has a natural combinatorial interpretation -- the coefficients count arm--leg inversions in increasing hook tableaux, which are pairs $(i,j)$ with $2\leq i<j$, where $i$ belongs to the row  and $j$ the column.
Using a map from $\I_k(N-r,1^r)$ to a set of standard hook tableaux that behaves nicely with respect to $K$-promotion, along with results of Reiner, Stanton and White \cite{rsw},
we exhibit a CSP for the set of increasing hook tableaux.

 \section{Enumerating increasing tableaux with Narayana numbers}
 
 We recall results concerning generalized Narayana numbers and generalized Schr\"oder numbers from \cite{sulankeD}.
 
 Let $\mathcal{C}(m,n)$ denote the set of lattice paths in $m$-dimensional space
 that run from $(0,0,\ldots,0)$ to $(n,n,\ldots,n)$ using the steps
 \[ X_1=(1,0,\ldots,0), \ X_2=(0,1,\ldots,0), \ \ldots, X_m=(0,0,\ldots,1) \]
 and lie in the region $\{ (x_1,x_2,\ldots,x_m) \mid 0 \leq x_m \leq x_{m-1} \leq \cdots \leq x_1\}$.
 A pair of steps $\epsilon_i\epsilon_{i+1}$ on a path $P=\epsilon_1\epsilon_2 \cdots \epsilon_{mn}$ is called
 an {\em ascent} if $\epsilon_i\epsilon_{i+1}=X_jX_r$ with $r < j$.
 The set of ascents on a path $P$ is denoted
 \[ \mbox{asc}(P)= \{ i \mid \epsilon_{i-1}\epsilon_i=X_jX_r \mbox{ for } r < j \}. \]
 
 \noindent For $m \geq 2$ and $0 \leq \ell \leq (m-1)(n-1)$, the $m$-Narayana number is defined as
 \[ N(m,n,\ell)=\Big{|} \Bigl\{P \in \mathcal{C}(m,n) \Big{|} \vert\mbox{asc}(P)\vert=\ell\Bigl\}\Big{|}. \]
 
 \noindent For $m \geq 2$ and $n \geq 1$, the $n$th $m$-Narayana polynomial is defined as
 \[ N_{m,n}(t)=\sum_{\ell =0}^{(m-1)(n-1)}N(m,n,\ell) t^\ell. \]
 
 The $m$-dimensional Catalan numbers are given by the sequence $(N_{m,n}(1))_{n \geq 0}$ and these enumerate $\SYT(m \times n)$.
 By the hook length formula, \[ N_{m,n}(1)=(mn)!\prod_{i=0}^{m-1}\frac{i!}{(n+i)!}. \]
 
 The {\em small} {\em $m$-Schr\"oder numbers} are given by the sequence $(N_{m,n}(2))_{n \geq 0}$.  In the case where $m=2$, Pechenik showed that the small $2$-Schr\"oder numbers enumerate the set of increasing $2 \times n$ tableaux.
 
 We will make use of the following proposition and corollary from \cite{sulankeD}.
 
 \begin{prop}\cite[Proposition 1]{sulankeD}
	For $m \geq 2$ and for $0 \leq \ell \leq(m-1)(n-1)$,
	$$N(m,n,\ell) = \sum_{j=0}^{\ell}(-1)^{\ell-j}\begin{pmatrix}mn+1\\\ell-j\end{pmatrix} \prod_{i=0}^{m-1}\begin{pmatrix}n+i+j\\n\end{pmatrix}\begin{pmatrix}n+i\\n\end{pmatrix}^{-1}.$$
\end{prop}

 \begin{cor}\cite[Corollary 1]{sulankeD}\label{sulankecor}  For $m \geq 2$ and $n \geq 1$, $N_{m,n}(t)$ is a self-reciprocal polynomial of degree $(m-1)(n-1)$.  In other words, for each $n$, the sequence of coefficients of $N_{m,n}(t)$ is symmetric.\end{cor}
 
A path $P = \epsilon_1\epsilon_2 \cdots \epsilon_{mn}\in \mathcal{C}(m,n)$ gives a standard $m \times n$ tableau
by reading the path left to right and placing $i$ in the $k$th row of the tableau whenever $\epsilon_i = X_k$.
The condition $0 \leq x_m \leq x_{m-1} \ldots \leq x_1$ ensures that if $\epsilon_i = X_k$,
then the number of occurrences of $X_{k-1}$ in the sequence occurring previously is strictly greater than
the number of occurrences of $X_k$, so the tableau generated by this procedure is standard.
In the case where $m=2$, $\mathcal{C}(2,n)$ consists of the paths from $(0,0)$ to $(n,n)$
with horizontal and vertical steps that stay below the line $y=x$.
This is a very well-known set of objects counted by the Catalan numbers.

An ascent occurs in a path $P$ precisely when the tableau generated by it has an entry $i$ occurring in a row above $i-1$
in the rectangular tableau it encodes.
It follows that $N(m,n,\ell)$ is equal to the number of tableaux in $\SYT(m \times n)$ for which an entry $i$
appears in a row above an $i-1$ exactly $\ell$ times.
For a tableau $T \in \SYT(\lambda)$, let $\mbox{asc}(T)=\{i \mbox{ in } T \mid i \mbox{ occurs in a row above } i-1\}$.
 
To obtain a counting formula for increasing tableaux of rectangular shape,
we define a map $\phi:\I_k(m \times n) \rightarrow \SYT(m \times n)$.
We first define $\phi_j : \I_j(m \times n) \rightarrow \I_{j-1}(m \times n)$, for $j \geq 1$.
For $T \in \I_j(m \times n)$, let $a$ be the minimal entry that appears more than once in $T$.
Increase all entries in $T$ that are greater than or equal to $a$, except for the leftmost value of $a$.
Define $\phi:\I_k(m \times n) \rightarrow \SYT(m \times n)$ as a composition
 $\phi=\phi_1 \circ \phi_{2} \circ \cdots \phi_{k-1} \circ \phi_{k}.$

\begin{exa}  Below we find the image of a tableau $T$ under $\phi:\I_3(3 \times 3) \rightarrow \SYT(3 \times 3)$.
\[ T=\vy{1 & 3 & 4 \cr 2 & 4 & 5 \cr 4 & 5 & 6 \cr} \overset{\phi_{3}}\longmapsto  \vy{1 & 3 & 5 \cr 2 & 5 & 6 \cr 4 & 6 & 7 \cr}
 \overset{\phi_{2}}\longmapsto \vy{1 & 3 & 6 \cr 2 & 5 & 7 \cr 4 & 7 & 8 \cr} \overset{\phi_{1}}\longmapsto 
 \vy{1 & 3 & 6 \cr 2 & 5 & 8 \cr 4 & 7 & 9 \cr}=\phi(T) \]

\noindent Note that $\phi$ is not one-to-one.
 For example, to determine the preimage of the tableau $S= \vy{1 & 3 & 6 \cr 2 & 5 & 8 \cr 4 & 7 & 9 \cr}$, \vspace{1mm}
 we consider $\mbox{asc}(S) = \{3,5,6,8\}$.
 Since $k=3$, each $3$-element subset $\{a,b,c\}$ of $\mbox{asc}(S)$, with $a<b<c$, corresponds to an element in the preimage
 by first subtracting one from all entries in $S$ that are greater than or equal to $c$,
 then subtracting one from all entries in the resulting tableau that are greater than or equal to $b$,
 and then repeating the process for $a$.
 So there are ${4 \choose 3}$ elements in the preimage of $S$; specifically
 \[ \phi^{-1}(S)=\left\{\vy{1 & 2 & 3 \cr 2 & 3 & 5 \cr 3 & 4 & 6 \cr}, \vy{1 & 2 & 4 \cr 2 & 4 & 5 \cr 3 & 5 & 6 \cr},
  \vy {1 & 2 & 4 \cr 2 & 3 & 5 \cr 3 & 5 & 6 \cr}, \vy{1 & 3 & 4 \cr 2 & 4 & 5 \cr 4 & 5 & 6 \cr}\right\}. \vspace{1mm} \]

\end{exa}

\begin{thm}\label{cardthm}For $k\geq 0$, $$ \vert\I_k(m\times n)\vert = \sum_{\ell=k}^{(m-1)(n-1)} \binom{\ell}{k} N(m,n,\ell).$$\end{thm}

\begin{proof} For any tableau $T \in \SYT(m \times n)$, $\phi^{-1}(T) \neq \emptyset$ if and only if
 $\vert \mbox{asc}(T) \vert \geq k$ and if $\vert \mbox{asc}(T)\vert=\ell \geq k$,
 then $\vert \phi^{-1}(T)\vert = \displaystyle {\ell \choose k}$.
 We have
\begin{eqnarray*} \vert \I_k(m \times n)\vert &=& \sum_{T \in \SYT(m \times n)} \vert \phi^{-1}(T) \vert \cr
 & = & \sum_{\ell=k}^{(m-1)(n-1)} {\ell \choose k } \Bigl\vert\Bigl\{T \in \SYT(m \times n) \Big{|} 
  \vert \mbox{asc}(T)\vert = \ell \Bigr\}\Bigr\vert\cr
 & = & \sum_{\ell=k}^{(m-1)(n-1)} {\ell \choose k} N(m,n,\ell). \end{eqnarray*} \vspace{-1mm} \end{proof}

\begin {cor}\label{incone} The number of increasing tableaux of shape $m \times n$ with exactly one repeated entry is given by 
\[ \displaystyle |\I_1(m \times n)| = \frac{(m-1)(n-1)}{2} \vert \SYT(m \times n) \vert. \]
\end {cor}

\begin{proof}

By Corollary \ref{sulankecor}, we have $N(m,n,\ell) = N(m,n,(m-1)(n-1) - \ell)$ for $0 \leq \ell \leq (m-1)(n-1)$. It follows that 
	\small{\begin{align*}
		 \vert \I_1(m \times n) \vert &= \frac{1}{2} \Big{(} \sum_{\ell=0}^{(m-1)(n-1)} \ell N(m,n,\ell) + \sum_{\ell=0}^{(m-1)(n-1)} ((m-1)(n-1) - \ell) N(m,n,\ell) \Big{)}\\
		 			   &= \frac{(m-1)(n-1)}{2} \sum_{\ell=0}^{(m-1)(n-1)} N(m,n,\ell) \\
					   &= \frac{(m-1)(n-1)}{2} \vert \SYT(m \times n)\vert.
 \end{align*} } \vspace{-1mm} \end{proof}

Using the above result, we can use the hook length formula for $\vert \SYT(m \times n) \vert$ to give the cardinality of
$\I_1(m \times n)$.

\begin{cor}  \label{33cor} For $m\geq 2$, the number of increasing tableaux of shape $m \times n$ with maximum entry $mn-1$
 is given by
 \[ \vert \I_1(m \times n) \vert = \frac{(m-1)(n-1)(mn)!}{2}\prod_{i=0}^{m-1}\frac{i!}{(n+i)!}. \]\end{cor}

Pechenik revealed a relationship between $\I_k(2 \times n)$ and small Schr\"oder numbers \cite[Theorem 1.1]{pechenik}.  The $n$th small Schr\"oder number is equal to $N_{2,n}(2)$ while the $n$th large Schr\"oder number is equal to $2N_{2,n}(2)$.  A {\em large Schr\"oder path} is a path from $(0,0)$ to $(n,n)$ with steps of the form $(1,0)$, $(0,1)$ and $(1,1)$ that stays below the line $y=x$.  A {\em small Schr\"oder path} is a large Schr\"oder path with no diagonal steps along $y=x$.  Pechenik's bijection (in a slightly modified form) between $\I_k(2 \times n)$ and small Schr\"oder paths is given by assigning a step $\epsilon_i$ to each entry $i$ in $T \in \I_k(2 \times n)$.  If $i$ appears only in the first row, then $\epsilon_i=(1,0)$, while if $i$ appears only in the second row, $\epsilon_i=(0,1)$ and if $i$ appears in both the first and second rows, then $\epsilon_i=(1,1)$.  This gives a small Schr\"oder path $P_T=\epsilon_1\epsilon_2\ldots \epsilon_{2n-k}$ and the procedure is easily reversible: given a small Schr\"oder path from $(0,0)$ to $(n,n)$, we can construct a tableau $T \in \I_k(2 \times n)$.

\begin{exa} We give the small Schr\"oder path for $T=\vy{1 & 3 & 4 & 5 \cr 2 & 4 & 5 & 6 \cr}$.
\psset{unit=8mm}
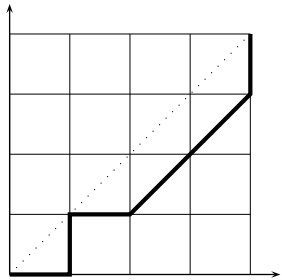
\begin{center}\begin{pspicture}(-1,-1)(5,5)
 \psset{linewidth=.5pt}
  \psline{->}(0,0)(4.5,0)\psline{->}(0,0)(0,4.5)
  \psline[linestyle=dotted](0,0)(4,4)
 \psset{linewidth=.2pt}
  \multips(0,0)(0,1){5}{\psline(4,0)}
  \multips(0,0)(1,0){5}{\psline(0,4)}
 \psset{linewidth=1.6pt}
  \psline(0,0)(1,0)(1,1)(2,1)(4,3)(4,4) 
\end{pspicture}\end{center}
 \end{exa}

\vspace{-.25cm}
We generalize Pechenik's result for rectangular increasing tableaux of arbitrary size, then define small $m$-Schr\"oder paths and give a bijection between these paths and the set of all increasing rectangular tableaux of shape $m \times n$.
\begin{cor} For $m \geq 2$ and $n \geq 1$ we have
$\displaystyle \sum_{k=0}^{(m-1)(n-1)}\vert \I_k(m \times n) \vert=N_{m,n}(2).$  In other words, the total number of increasing tableaux of shape $m \times n$  is given by the small $m$-Schr\"oder number.\end{cor}

\begin{proof}
We have \begin{eqnarray*}\sum_{k=0}^{(m-1)(n-1)}\vert \I_k(m \times n) \vert&=&\sum_{k=0}^{(m-1)(n-1)} \sum_{\ell=k}^{(m-1)(n-1)} {\ell \choose k} N(m,n,\ell)\\ \\
&=& \sum_{t=0}^{(m-1)(n-1)} \left(\sum_{i=0}^t {t \choose i} \right)N(m,n,t)\\ \\
&=& \sum_{t=0}^{(m-1)(n-1)} 2^t \ N(m,n,t) =N_{m,n}(2). \end{eqnarray*}\end{proof}

Sulanke defined {\em large $m$-Schr\"oder paths} \cite{sulankeD} as paths running from $(0,0,\ldots,0)$ to $(n,n,\ldots,n)$
with nonzero steps of the form $(\xi_1,\xi_2,\ldots,\xi_m)$, with $\xi_i \in \{0,1\}$,
that lie in the region $\{(x_1,x_2,\ldots,x_m)\mid 0 \leq x_m \leq x_{m-1} \leq \cdots \leq x_1\}$.
He proved that the number of large $m$-Schr\"oder paths is equal to $2^{m-1}N_{m,n}(2)$.  

We define a {\em small $m$-Schr\"oder path} to be a large $m$-Schr\"oder path with the property that the path does not contain
any steps from $(x_1,\ldots,x_{j-1},a,a, x_{j+2},\ldots,x_m)$ to $(y_1,\ldots,y_{j-1},a+1,a+1,y_{j+2},\ldots,y_m)$.
In other words, if after $k$ steps the path reaches position $(x_1,\ldots,x_m)$, where $x_j=x_{j+1}$,
then the $(k+1)$th step $\epsilon_{k+1}=(\xi_1,\ldots, \xi_{m})$ cannot have $\xi_j = \xi_{j+1} = 1.$
For example, a small $3$-Schr\"oder path is a path from $(0,0,0)$ to $(3,3,3)$ with nonzero steps of the form
 $(\xi_1,\xi_2,\xi_3)$, $\xi_i \in \{0,1\}$, that lies in the region $\{(x,y,z)\mid 0 \leq z \leq y \leq x\}$
and does not contain any steps from $(a,a,z)$ to $(a+1,a+1,z^\prime)$ or from $(x,b,b)$ to $(x^\prime,b+1,b+1)$.
In the case where $m=2$, the small $m$-Schr\"oder paths are the usual small Schr\"oder paths.

\begin{thm}\label{schroderthm}  There is a bijection between the collection of small $m$-Schr\"oder paths and the set of all increasing tableaux of shape $m \times n$.\end{thm}

\begin{proof}

For an increasing tableau $T$ with largest entry $mn -k$, define a path $P_T=\epsilon_1\epsilon_2\cdots \epsilon_{mn-k}$ from $(0,0,\ldots,0)$ to $(n,n,\ldots,n)$ in the following way.  For each $1 \leq i \leq mn-k$, let $\epsilon_i=(\xi_1,\xi_2,\ldots,\xi_m)$ where $\xi_j=1$ if $i$ appears in the $j$th row of $T$ and $\xi_j=0$ otherwise.  Since $T$ has strictly increasing columns, $P_T$ lies in the region $\{(x_1,x_2,\ldots,x_m)\mid 0 \leq x_m \leq x_{m-1} \leq \cdots \leq x_1\}$.  If, after the $k$th step $\epsilon_k$, the path reaches position $t_k$ then, after the $(k+1)$th step, the path reaches position $t_{k+1}=t_k+\epsilon_{k+1}$. Furthermore, the subtableau of $T$ of shape $\lambda_k=t_k=(x_1,x_2,\ldots,x_m)$ is the portion of $T$ that contains the entries $1,2,\ldots,k$.  (See Example \ref{schroderex} for an illustration.)  If $t_k=(x_1,\ldots,x_m)$ has $x_j=x_{j+1}$, then the subtableau of $T$ containing the entries up to and including $k$ has $x_j$ boxes in both the $j$th and $(j+1)$th rows.  If $\epsilon_{k+1}=(\xi_1,\ldots,\xi_m)$ has $\xi_j=\xi_{j+1}=1$ then $t_{k+1}=(y_1,\ldots,x_{j}+1,x_j+1,\ldots,y_m)$ so the subtableau of $T$ containing the entries up to and including $k+1$ has $x_j+1$ boxes in both the $j$th and $(j+1)$th row, which forces two entries equal to $k+1$ in column $x_j+1$ of $T$.  It follows that $P_T$ is a small $m$-Schr\"oder path.

Given a small $m$-Schr\"oder path $P_T$, we can construct an increasing tableau $T$ of shape $m \times n$ by reversing the above procedure.\end{proof}

\begin{exa}\label{schroderex}
\
\vspace{2mm}

  \noindent For $T=\vy{1 & 2 & 4 & 5 \cr 2 & 4 & 5 & 7 \cr 3 & 6 & 9 & 10 \cr 4 & 8 & 10 & 11 \cr}$, \vspace{2mm}
 $P_T=\epsilon_1\epsilon_2\cdots \epsilon_{11}$ where
$\epsilon_1=(1,0,0,0)$, $\epsilon_2=(1,1,0,0)$, $\epsilon_3=(0,0,1,0)$, $\epsilon_4=(1,1,0,1)$, $\epsilon_5=(1,1,0,0)$, etc.
 The steps in $P_T$ take the path to positions $t_1=(1,0,0,0)$, $t_2=(2,1,0,0)$, $t_3=(2,1,1,0)$, $t_4=(3,2,1,1)$,
 $t_5=(4,3,1,1)$, etc.
The position $t_i$ gives the shape $\lambda=t_i$ of the subtableau of $T$ that contains the entries $1, 2 \ldots, i$.

\end{exa}

\begin{rem}
Using the same construction as in the proof of Theorem \ref{schroderthm}, the large $m$-Schr\"oder paths are in one-to-one correspondence with the set of row-increasing tableaux of shape $m \times n$ where the entries are an initial segment of $\mathbb{Z}_{\geq 0}$ or, by transpose, to the set of semistandard $n \times m$ tableaux with entries an initial segment of $\mathbb{Z}_{\geq 0}$.  By \cite[Proposition 10]{sulankeD}, this subset of the collection of semistandard $n \times m$ tableaux has cardinality equal to $2^{m-1}N_{m,n}(2)$.

\end{rem}

\section{Cyclic sieving phenomena}\label{sec:csp}

In this section, we give a CSP for increasing hook tableaux.
A CSP for semistandard hook tableaux was given in \cite{bms}.
We also show that the polynomial obtained by taking the natural $q$-analogue of the integer in Corollary \ref{33cor},
along with K-jeu de taquin promotion does not, in general, give a CSP for increasing rectangular tableaux,
apart from the $2 \times n$ version given in \cite{pechenik}.

Our focus is on increasing hook tableaux and for such tableaux, $K$-promotion, 
which defines a bijection $\partial:\I_k(N-r,1^r) \rightarrow \I_k (N-r,1^r)$, can be described in the following way.
Given $T \in \I_k(N-r,1^r)$, replace the $1$ in $T$ with a dot and repeatedly move all dots through the tableau
using the rules below until every dot appears in the right-most box of the row or the lowest box in the column. 
Then replace each dot with $N-k$ and decrease all other entries in the tableau by one to obtain $\partial(T)$.

\begin{equation}\label{jdtmoves}\vy{\bullet & a \cr b \cr } \rightarrow
\left\{\begin{array}{ll}
\vy{a & \bullet \cr  b \cr} & \text{if $a < b$} \\[5mm]
\vy{b & a \cr \bullet \cr} & \text{if $b<a $.}
\end{array}\right. , \ \ \ \vy{ \bullet & a \cr  a \cr} \rightarrow \vy{a & \bullet \cr \bullet \cr}. \ \ \ 
\end{equation}

\bigskip

Note that when $k=0$, $K$-promotion amounts to Sch\"utzenberger's jeu de taquin promotion on $\SYT(N-r,1^r)$.
For a description of $K$-promotion on increasing tableaux for general shapes, see \cite{pechenik}.
For a more general description of promotion and its basic properties, a survey is given in~\cite{stanley}.

\begin{exa} For $T=\vy{1 & 2 & 4 & 5  \cr 2 \cr 3 \cr 5 \cr}$,  $K$-promotion works as follows:
\scriptsize{$$T \rightarrow \vy{ \bullet & 2 & 4 & 5  \cr 2 \cr 3 \cr 5 \cr} \rightarrow \vy{2 & \bullet & 4 & 5  \cr \bullet \cr 3 \cr 5 \cr} \rightarrow \vy{2 & 4 & 5  & \bullet \cr 3 \cr 5 \cr \bullet \cr} \rightarrow \vy{1 & 3 & 4 &  5 \cr 2 \cr 4 \cr 5 \cr}=\partial(T).$$}
\end{exa}




The {\em content} of a tableau $T \in \I_k(\lambda)$, where $\lambda$ is a partition of $N$,
is equal to $\alpha=(\alpha_1,\ldots,\alpha_{N-k})$, where $\alpha_i$ gives the number of entries equal to $i$ in $T$;
we denote this by $\cont(T)$.
The symmetric group $S_{N-k}$ on $N-k$ letters acts on $(N-k)$-tuples by permuting places:
 \[ \theta  (\alpha_1,\alpha_2,\ldots,\alpha_{N-k})=(\alpha_{\theta(1)},\alpha_{\theta(2)},\ldots,\alpha_{\theta(N-k)}),
  \mbox{ where } \theta \in S_{N-k}. \]  
For increasing hook tableaux, $K$-promotion permutes the content via the cycle $\sigma=(2\ 3\cdots N-k) \in S_{N-k}$.
In other words, if $\cont(T)=(\alpha_1,\alpha_2,\ldots,\alpha_{N-k})$, where $\alpha_1$ is necessarily equal to $1$, then
\begin{equation}\label{contentperm} \cont(\partial(T)) = (1, \alpha_3,\ldots, \alpha_{N-k}, \alpha_{2}) =
 \sigma (\alpha_1,\ldots,\alpha_{N-k}),\end{equation}
 for $T\in \I_k(N-r,1^r)$.

The cardinality of $\I_k(N-r,1^r)$ is given by
\[ \vert \I_k(N-r,1^r)\vert ={N-k-1 \choose r}{r \choose k}. \]

To give  a CSP for $\I_k(N-r,1^r)$, we work with a map $\psi:\I_k(N-r,1^r) \rightarrow \SYT(N-r-k,1^{r})$ that behaves nicely with respect to $K$-promotion. This will allow us to use established results concerning $\SYT(N-r-k,1^{r})$.
Define $\psi: \I_k(N-r,1^r)\rightarrow \SYT(N-r-k,1^{r})$ by deleting the $k$ entries in the row of $T \in \I_k(N-r,1^r)$ that also appear in the column of $T$.  Then $\psi$ is onto, but not one-to-one.  The following lemma follows easily from (\ref{jdtmoves}).

\begin{lem} \label{commutethm} If $T \in \I_k(N-r,1^r)$, then $\psi(\partial(T))=\partial(\psi(T))$. \end{lem}





The {\em order of promotion} on $\I_k(N-r,1^r)$ is the smallest positive integer $\ell$ that satisfies $\partial^\ell(T)=T$ for all $T \in \I_k(N-r,1^r)$.  When $k=0$, there is a one-to-one correspondence between $\SYT(N-r,1^r)$ and the set $\mathcal{A}$ that consists of subsets of $\{2,\ldots,N\}$ containing $r-1$ elements.  Define $\gamma:\SYT(N-r,1^r) \rightarrow \mathcal{A}$, by defining $\gamma(S)$ to be the set of entries in the first column of $S$ that sit below the $(1,1)$-box and let $\theta=(2 \ 3 \ 4 \  \cdots N)^{-1} \in S_N$.  We have $\gamma(\partial(S))=\theta(\gamma(S))$ for $S\in \SYT(N-r,1^r)$ so jeu de taquin promotion on $S  \in \SYT(N-r,1^r)$ is completely determined by considering the action of $\theta$ on the column of $S$.  It follows that the order of promotion on $\SYT(N-r,1^r)$ is equal to $N-1$.  

\begin{thm}The order of promotion on $\I_k(N-r,1^r)$ is equal to $N-k-1$. \end{thm}
\begin{proof}
Let $\sigma=(2\ 3\cdots N-k) \in S_{N-k}$ and suppose that $T \in  \I_k(N-r,1^r)$ has
content $\alpha=(\alpha_1,\alpha_2,\ldots,\alpha_{N-k})$. 
Then the content of $\partial^{N-k-1}(T)$ is equal to $\sigma^{N-k-1}(\alpha_1,\alpha_2,\ldots, \alpha_{N-k})=\alpha$.   

We have $\partial^{N-k-1}(S)=S$ for $S \in \SYT(N-r-k,1^{r})$ so $\partial^{N-k-1}(\psi(T))=\psi(T)$ for $T \in \I_k(N-r,1^r)$.    By Lemma \ref{commutethm}, $\psi(\partial^{N-k-1}(T))=\psi(T)$ and since $\cont(T)=\cont(\partial^{N-k-1}(T))$, we have $\partial^{N-k-1}(T)=T$. Furthermore, $T \in \I_k(N-r,1^r)$ with content equal to $(1,\underbrace{2,2,\ldots,2}_k,\underbrace{1,\ldots,1}_{N-k-1})$ is fixed by no less than $N-k-1$ iterations of $K$-promotion. \end{proof}

The following theorem is due to Reiner, Stanton and White \cite{rsw},
where the theorem is stated in terms of $k$-subsets of $\{1,2,\ldots,N\}$
under the action of the long cycle $(1\ 2\cdots N) \in S_N$.

\begin{thm}\label{rswthm}\cite[Theorem 1.1]{rsw} 
 The triple $\left(\SYT(N-r,1^r), C, X_0(q) \right)$ satisfies the cyclic sieving phenomenon,
 where $C$ is the cyclic group of order $N-1$ given by jeu de taquin promotion on $\SYT(\lambda)$
 and $X_0(q)=\left[\begin{array}{c} N-1 \cr r \cr \end{array}\right]_{q}.$

\end{thm}

\bigskip

\n Let $f_1(q)=\left[\begin{array}{c}N -k-1 \cr r \cr \end{array} \right]_q$,
 $f_2(q)=\left[\begin{array}{c}r \cr k \cr \end{array} \right]_q$ and $X(q)=f_1(q) f_2(q)$, \vspace{1mm}
 which is a $q$-analogue of the formula that enumerates $\I_k(N-r,1^r)$.
In fact, $X(q)$ has a fairly natural combinatorial interpretation.
An ordered pair $(i,j)$ with $2\le i<j \le N-k$ will be called an \emph{inversion} in $T\in\I_k(N-r,1^r)$
if $i$ appears as a row entry in $T$ and $j$ appears as a column entry in $T$.
Then $\sum_{T} q^{a(T)}=q^{k \choose 2} X(q),$
where $\lambda=(N-r,1^r)$ and $a(T)$ is the number of inversions in $T$.
This follows easily from the interpretation of the $q$-binomial coefficients (or Gaussian coefficients)
as generating functions for subsets with respect to ``between-set inversions''.
Details of this interpretation are given in~\cite{GJ}.

\bigskip

\n Let $\omega$ be a primitive $(N-k-1)$th root of unity.  Then $\omega^m$ is a primitive $d$th root of unity where $d\cdot \mbox{gcd}(N-k-1,m)=N-k-1$ and by \cite[Proposition 4.2]{rsw}, 
\begin{equation}\label{order1}f_1(\omega^m)=\left\{ \begin{array}{ll}  \left( \begin{array}{c} (N-k-1)/d \cr r/d \cr \end{array} \right) & \mbox{ if } d \vert r \cr
\cr
0 & \mbox{ otherwise. } \cr \end{array}\right. \end{equation}

In general, $f_2(\omega^m)$ may not be an integer but we are only concerned with the value of $f_2(\omega^m)$ when $f_1(\omega^m) f_2(\omega^m) \neq 0$.  In particular, if $f_1(\omega^m) \neq 0$, then $d \vert r$ so we have
\begin{equation}\label{order2} f_2(\omega^m)=\left\{ \begin{array}{ll}  \left( \begin{array}{c} r/d \cr k/d \cr \end{array} \right) & \mbox{ if } d \vert k \cr \cr

0 & \mbox{ otherwise, } \cr \end{array}\right.  \end{equation}
when $f_1(\omega^m) \neq 0$.

Lemma \ref{contlem} is the main ingredient that will be used to prove a CSP for $\I_k(N-r,1^r)$.  The following example will be useful when reading the proof of Lemma \ref{contlem}.

\begin{exa}  Consider $\psi:\I_2(5,1^6) \rightarrow \SYT(3,1^6)$. Promotion on a tableau in $\SYT(3,1^6)$ corresponds to the action of the permutation $\theta = (2\ 9\ 8\ 7\ 6\ 5\ 4\ 3)$ on the column entries of the tableau.  Since $\theta^4 = (2\ 6)(3\ 7)(4\ 8)(5\ 9)$, the column of a tableau in $\SYT(3,1^6)$ is fixed by $\theta^4$ only when the entries in the column of the tableau below the $(1,1)$-box correspond to the values in three of the four 2-cycles in the decomposition of $\theta^4$. The following tableau in $\SYT(3,1^6)$ satisfies $\partial^4(S)=S$:
 \[ S=\vy{1 & 4 & 8 \cr 2 \cr 3 \cr 5 \cr 6 \cr 7 \cr 9 \cr}. \]
There are $\binom{6}{2} = 15$ tableaux that map to $S$ under $\psi$.
We wish to determine those tableaux in the preimage of $S$ with content that is fixed by $\partial^4$.
In general, a tableau $T \in \I_2(5,1^6)$ has content that is fixed by $\partial^4$ if and only if
the content of $T$ is equal to one of the following:
 \[ (1,\!2,\!1,\!1,\!1,\!2,\!1,\!1,\!1), (1,\!1,\!1,\!1,\!2,\!1,\!1,\!1,\!2),
  (1,\!1,\!2,\!1,\!1,\!1,\!2,\!1,\!1), (1,\!1,\!1,\!2,\!1,\!1,\!1,\!2,\!1). \]   
 If $T \in \I_2(5,1^6)$ also satisfies $\psi(T)=S$, then the two elements in the row of $T$ that are repeated in the column
must belong to one of the $2$-cycles that appear in the decomposition of $\theta^4$,
 so if $\psi(T)=S$ and $\cont(T)=\cont(\partial^4(T))$ then $\cont(T)$ must be equal to one of the first three sequences above.
This completely determines those tableaux in the preimage of $S$ that have content fixed by $\partial^4$:  
 \[ \vy{1 & 2 & 4 & 6 &8  \cr 2 \cr 3 \cr 5 \cr 6 \cr 7 \cr 9 \cr},\quad \vy{1 & 3 & 4 & 7 &8  \cr 2 \cr 3 \cr 5 \cr 6 \cr 7 \cr 9 \cr},
  \quad  \vy{1 & 4 & 5 & 8 &9  \cr 2 \cr 3 \cr 5 \cr 6 \cr 7 \cr 9 \cr}. \]
\end{exa}
\bigskip
\begin{lem}\label{contlem}Let $S \in \SYT(N-r-k,1^{r})$ with $\partial^m(S) = S$ and suppose that $\omega$ is a primitive $(N-k-1)$th root of unity.
The number of $T\in \I_k(N-r,1^r)$ with $\psi(T)=S$ such that $\cont(T)=\cont(\partial^m(T))$ is equal to
 $f_2(\omega^m)=\left[\begin{array}{c}r \cr k \cr \end{array} \right]_{q=\omega^m}.$\end{lem}

\begin{proof}
	Since $\omega$ is a primitive $(N-k-1)$th root of unity, $\omega^m$ is a $d$th root of unity where $d\cdot \mbox{gcd}(N-k-1,m)=N-k-1$.  Let $\theta=\sigma^{-1}=(2\ 3\ 4\ 5\cdots N-k)^{-1} \in S_{N-k}$.  The column of $\partial(S)$ is given by the action of $\theta$ on the entries of the column of $S$ that sit below the $(1,1)$-box, so $\theta^m$ fixes these $r$ elements in the column of $S$.  We can write $\theta^m=\theta_1\theta_2\cdots \theta_{m^\prime}$, which is a product of $m^\prime=gcd(N-k-1,m)$ disjoint cycles of length $d$.   Since $\theta^m$ fixes the $r$ column elements of $S$, we have that $d$ divides $r$.  
	
Let $T\in \I_k(N-r,1^r)$ have content equal to $\alpha=(\alpha_1, \alpha_2, \ldots, \alpha_{N-k})$,
 and suppose that $\cont(\partial^m(T))=\alpha$.
Then by (\ref{contentperm}), $\sigma^m \alpha=\alpha$ and since $\sigma^m$ is the product of $d$-cycles and there are exactly $k$ entries $\alpha_i$ that are equal to $2$, we have that $d$ divides $k$.  Furthermore, the  $k$ repeated entries in the row of $T$ can be partitioned into $k/d$ sets of size $d$, where each set consists of elements from one of the $d$-cycles $\theta_1,\theta_2,\ldots,\theta_{m^\prime}$ in the decomposition of $\theta=\sigma^{-1}$.

Since $\theta^m$ fixes the entries in the column of $S$, the entries below the $(1,1)$-box must consist of
the values from $\ell=r/d$ of the $d$-cycles $\theta_1,\theta_2,\ldots,\theta_{m^\prime}$.
Denote this subset of $d$-cycles that give the column of $S$ by $\theta^\prime_1,\theta^\prime_2,\ldots,\theta^\prime_\ell$ .
If $T\in \I_k(N-r,1^r)$, with $\psi(T)=S$ and $\cont(T)=\cont(\partial^m(T))$,
then the $k$ entries in the row of $T$ that are repeated in the column can be partitioned into $k/d$ sets of size $d$,
where each set consists of elements from one of the $d$-cycles $\theta^\prime_1,\ldots,\theta^\prime_\ell$.
There are exactly $\left(\begin{array}{c} r/d \cr k/d \cr \end{array} \right) $ such tableaux and by (\ref{order2}), this is equal to $f_2(\omega^m)$.  \end{proof}

\begin{thm}  The triple $\left(\I_k(N-r,1^r), C, X(q)\right)$ satisfies the cyclic sieving phenomenon, where $C$ is the cyclic group of order $N-k-1$ given by $K$-promotion on $\I_k(N-r,1^r)$ and $X(q)=\left[\begin{array}{c}N -k-1 \cr r \cr \end{array} \right]_q  \left[\begin{array}{c}r \cr k \cr \end{array} \right]_q$.

\end{thm}

\begin{proof} Let $X=\{ T \in \I_k(N-r,1^r) \mid \partial^m(T)=T\}$ and
 \[ Y=\{ T \in \I_k(N-r,1^r) \mid \partial^m(\psi(T))=\psi(T) \mbox{ and } \cont(T)=\cont(\partial^m(T))\}. \]
If $T \in X$, then $\partial^m(T)=T$ so $\cont(\partial^m(T))=\cont(T)$ and $\psi(\partial^m(T))=\psi(T)$.
By Lemma \ref{commutethm}, $\partial^m(\psi(T))=\psi(T)$.
If $T \in Y$, then $\partial^m(\psi(T))=\psi(T)$ so $\psi(\partial^m(T))=\psi(T)$.
Since $\cont(\partial^m(T))=\cont(T)$, $\partial^m(T)=T$.
Thus $\vert X \vert = \vert Y \vert$ and by Theorem \ref{rswthm} and Lemma \ref{contlem},
$\vert Y \vert = f_1(\omega^m) f_2(\omega^m)$.
\end{proof}

We close with an example that shows that a natural $q$-analogue of the polynomial in Corollary \ref{33cor},
coupled with $K$-jeu de taquin promotion does not give a CSP for $\I_1(3 \times 3)$.

\begin{exa}Let $X(q)=\displaystyle \frac{(q^9-1)(q^8-1)(q^7-1)(q^6-1)}{(q^4-1)(q^3-1)^2(q-1)}$, \vspace{1mm}
which is a natural $q$-analogue of the integer from Corollary \ref{33cor} for $n=3$.
The order of promotion on $\I_1(3 \times 3)$ is equal to $8$ and there are four tableaux in $\I_1(3 \times 3)$
that have order equal to two. 
(See \cite{pechenik} for the definition of $K$-promotion for rectangular shapes.)
However, if $\omega$ is a primitive eighth root of unity, $X(\omega^2) = 2-2i$ is not even an integer.

\end{exa}




\end{document}